\documentclass[11pt]{amsart}

\usepackage{epigamath}

\usepackage[english]{babel}

\numberwithin{equation}{section}

\usepackage[shortlabels]{enumitem}
\setlist[enumerate,1]{label={\rm(\arabic*)}, ref={\rm\arabic*}} 

\usepackage{mabliautoref}
\usepackage{amssymb,amsthm,amsmath}
\usepackage{mathtools}
\usepackage[all]{xy}
\usepackage{tikz}
\usepackage{tikz-cd}
\usetikzlibrary{decorations.markings}
\usepackage{esint}
\usepackage{faktor}
\usepackage{chngcntr}
\usepackage{minibox}
\usepackage[framemethod=TikZ]{mdframed}

\usepackage{xparse}

\NewDocumentCommand\UpArrow{O{2.0ex} O{black}}{%
    \mathrel{\tikz[baseline] \draw [line width=0.5pt, decoration={markings,mark=at position 1 with {\arrow[scale=2, line width=0.25pt]{to}}},  postaction={decorate}, #2] (0,0) -- ++(0,#1);}
 }

\newcommand{\expl}[2]{\underset{\mathclap{\minibox[c]{$\UpArrow[10pt]$\\ \fbox{\footnotesize #2}}}}{#1}}

\newcommand{\explshift}[3]{\underset{\mathclap{\minibox[c]{$\UpArrow[10pt]$\\ \hspace{#1} \fbox{\footnotesize #3}}}}{#2}}

\newcommand{\bA}{\mathbb{A}}

\newcommand{\bF}{\mathbb{F}}

\newcommand{\bN}{\mathbb{N}}

\newcommand{\bP}{\mathbb{P}}
\newcommand{\bQ}{\mathbb{Q}}

\newcommand{\bZ}{\mathbb{Z}}

\newcommand{\scr}{\mathcal}
\newcommand{\sA}{\scr{A}}

\newcommand{\sC}{\scr{C}}

\newcommand{\sE}{\scr{E}}
\newcommand{\sF}{\scr{F}}

\newcommand{\sH}{\scr{H}}

\newcommand{\sO}{\scr{O}}

\DeclareMathOperator{\characteristic}{{char}}

\DeclareMathOperator{\sstab}{{ss}}

\DeclareMathOperator{\Tr}{Tr}

\DeclareMathOperator{\coker}{{coker}}

\DeclareMathOperator{\Ker}{{Ker}}

\DeclareMathOperator{\Id}{{Id}}
\DeclareMathOperator{\im}{{im}}

\DeclareMathOperator{\length}{{length}}

\DeclareMathOperator{\rel}{{rel}}

\DeclareMathOperator{\reg}{reg}

\DeclareMathOperator{\rk}{{rk}}

\DeclareMathOperator{\Spec}{{Spec}}

\newcommand{\factor}[2]{\left. \raise 2pt\hbox{\ensuremath{#1}} \right/
        \hskip -2pt\raise -2pt\hbox{\ensuremath{#2}}}

\newcommand{\lra}{\longrightarrow}

\newcommand{\supth}[1]{\ensuremath{#1^{\mathrm{th}}}}

\EpigaVolumeYear{9}{2025} \EpigaArticleNr{7} \ReceivedOn{July 17, 2023}
\InFinalFormOn{July 24, 2024}
\AcceptedOn{August 13, 2024}

\title{Pseudo-effectivity of the relative canonical divisor and uniruledness in positive characteristic}
\titlemark{Pseudo-effectivity of the relative canonical divisor and uniruledness}

\author{Zsolt Patakfalvi}
\address{\'Ecole Polytechnique F\'ed\'erale de Lausanne (EPFL), MA C3 635, Station 8, 1015 Lausanne, Switzerland}
\email{zsolt.patakfalvi@epfl.ch}

\authormark{Zs.~Patakfalvi}

\AbstractInEnglish{We show that if $f\colon X \to T$ is a surjective morphism between smooth projective varieties over an algebraically closed field $k$ of  characteristic $p>0$ with geometrically  integral and non-uniruled generic fiber, then $K_{X/T}$ is pseudo-effective. 

The proof is based on covering $X$ with rational curves, which gives a contradiction as soon as both the base and the generic fiber are not uniruled. However, we assume only that the generic fiber is not uniruled. Hence, the hardest part of the proof is to show that there is a finite smooth non-uniruled cover of the base for which we show the following: If $T$ is a smooth projective variety over $k$ and $\sA$ is an ample enough line bundle, then a cyclic cover of degree $p \nmid d$ given by a general element of $\left|\sA^d\right|$ is not uniruled. For this we show  the following cohomological uniruledness condition, which might be of independent interest: A smooth projective variety $T$ of dimenion $n$ is not uniruled whenever the dimension of the semi-stable part of $H^n(T, \sO_T)$ is greater than that of $H^{n-1}(T, \sO_T)$. 

Additionally, we also show singular versions of all the above statements.}

\MSCclass{14E99, 14G17, 14J40, 14F99}
\KeyWords{Semi-positivity, relative canonical divisor, uniruledness, subadditivity of Kodaira dimension}

\acknowledgement{The author was partially supported by the following grants: grant \#200021/169639 from the Swiss National Science Foundation,  ERC Starting grant \#804334.}

\begin{document}


\maketitle

\begin{prelims}

\DisplayAbstractInEnglish

\bigskip

\DisplayKeyWords

\medskip

\DisplayMSCclass

\end{prelims}


\newpage

\setcounter{tocdepth}{1}

\tableofcontents


\section{Introduction}

\emph{The base-field $k$ is algebraically closed and of characteristic $p>0$, unless otherwise stated.} 

Consider a fibration $f \colon X \to T$ between smooth projective varieties. Over characteristic zero ground-fields, the following statement has been known for a while:
\begin{equation}
\label{eq:char_0_statement}
\parbox{400pt}{  $K_{X/T}$ is pseudo-effective whenever one of the following two equivalent conditions has been met for the geometric generic fiber $X_{\overline{\eta}}$:
\begin{enumerate}[leftmargin=45pt]
\item[(Psef)]  
$K_{ X_{\overline{\eta}}}$ is pseudo-effective. 
\item[(N-ur)]  $X_{\overline{\eta}}$ is not uniruled.
\end{enumerate}
} 
\end{equation}
For a precise reference, we refer to  \cite[Theorem~4.1]{Nakayama_Noboru_Zariski_decomposition_and_abundance}, but the statement has been implicit already in the works of Viehweg; see, \textit{e.g.}, \cite{Viehweg_Weak_positivity}. Statements  stating different (semi-)positivity properties of $K_{X/T}$   have been used extensively  in characteristic zero algebraic geometry, for example to questions such as 
\begin{itemize} 
\item subadditivity of Kodaira dimension (here there are particularly many references, with the few initial ones being \cite{Fujita_On_Kahler_fiber_spaces,
Kawamata_Characterization_of_abelian_varieties,
Viehweg_Canonical_divisors_and_the_additivity_of_the_Kodaira_dimension_for_morphisms_of_relative_dimension_one,
Viehweg_Weak_positivity,
Kollar_Subadditivity_of_the_Kodaira_dimension}),
\item construction of moduli spaces of $K$-/KSBA-stable varieties (see, \textit{e.g.}, \cite{Viehweg_Weak_positivity_and_the_stability_of_certain_Hilbert_points,Kollar_Shepher_Barron_Threefolds_and_deformations,
Fujino_Semi_positivity_theorems_for_moduli_problems,
Kovacs_Patakfalvi_Projectivity_of_the_moduli_space_of_stable_log_varieties_and_subadditvity_of_log_Kodaira_dimension,
Codogni_Patakfalvi_Positivity_of_the_CM_line_bundle_for_families_of_K-stable_klt_Fanos,
Xu_Zhuang_On_positivity_of_the_CM_line_bundle_on_K-moduli_spaces}),
\item hyperbolicity questions (see, \textit{e.g.},\cite{Viehweg_Zuo_On_the_Brody_hyperbolicity_of_moduli_spaces_for_canonically_polarized_manifolds,
Caporaso_Harris_Mazur_Uniformity_of_rational_points,  Abramovich_Uniformity_of_stably_integral_points_on_elliptic_curves}), \item non-vanishing conjecture (see, \textit{e.g.}, \cite{Zhang_Abundance_for_3-folds_with_non-trivial_Albanese_maps_in_positive_characteristic}), \item geography of varieties (e.g, \cite{Chen_Chen_Jiang_The_Noether_inequality_for_algebraic_3-folds}), 
\item etc.
\end{itemize}
Motivated in part by the above applications, (semi-)positivity of the relative  canonical bundle has been an active area of research in the past ten years also over  fields $k$ of  characteristic $p>0$; see, \textit{e.g}.,  \cite{Patakfalvi_Semi_positivity_in_positive_characteristics,
Patakfalvi_On_subadditivity_of_Kodaira_dimension_in_positive_characteristic_over_a_general_type_base,
Chen_Zhang_The_subadditivity_of_the_Kodaira-dimension_for_fibrations_of_relative_dimension_one_in_positive_characteristics,
Ejiri_Direct_images_of_pluricanonical_bundles_and_Frobenius_stable_canonical___rings_of_generic_fibers,
Ejiri_Zhang_Iitaka_s_C_n_m_conjecture_for_3-folds_in_positive_characteristic,
Birkar_Chen_Zhang_Iitaka_s_C_n_m_conjecture_for_3-folds_over_finite_fields}. In these results, extra conditions are imposed compared to characteristic zero to exclude wild behavior in positive characteristics. 
By now, it is known that these extra conditions are necessary, as  statement \autoref{eq:char_0_statement} with the assumption (Psef)  is known to fail; see  \cite{Cascini_Ejiri_Kollar_Zhang_Subadditivity_of_Kodaira_dimension_does_not_hold_in_positive___characteristic}. 

Additionally, assumptions (Psef) and (N-ur) of \autoref{eq:char_0_statement} are not equivalent in positive characteristic, and (N-ur) does take into account some of the typical wild behavior; see the introduction of \cite{Patakfalvi_Zdanowicz_Ordinary_varieties_with_trivial_canonical_bundle_are_not_uniruled}. Hence, one could hope that statement \autoref{eq:char_0_statement} with condition (N-ur) still holds in positive characteristic, which is exactly our main theorem. 

\begin{theorem}[Smooth case of \autoref{thm:main}]
\label{thm:main_intro}
Let 
$f \colon X \to T$ be a surjective morphism between smooth projective varieties over $k$ with integral and non-uniruled geometric generic fiber. 
Then $K_{X/T}$ is pseudo-effective. 
\end{theorem}

Note that \autoref{thm:main_intro} implies that the examples of \cite{Cascini_Ejiri_Kollar_Zhang_Subadditivity_of_Kodaira_dimension_does_not_hold_in_positive___characteristic} need to have uniruled geometric generic fiber, and indeed they are (very) singular rational  curves. Also, as indicated in the statement of \autoref{thm:main_intro}, and as will also be the case for \autoref{thm:non_uniruled} and \autoref{cor:cyclic_cover_intro}, we state a singular version in the later parts of the article, allowing as bad singularities as the proof lets us do. This means complete intersection, $W \sO$-rational, and  $W \sO$-rational and Cohen-Macaulay singularities in the three respective cases. 

Our last remark concerning \autoref{thm:main_intro} is that the pseudo-effectivity of $K_{X/T}$ is one of  the weakest possible semi-positivity properties. For example, under the same assumption, $f_* \omega_{X/T}$ is known to be not always semi-positive. In fact, there are examples of  families $f \colon X \to T$ of smooth hyperbolic curves such that $f_* \omega_{X/T}$ is not nef; see \cite{Moret_Bailly_Familles_de_courbes_et_de_varietes_abeliennes_sur_P_1_II_exemples}. Nevertheless, from \autoref{thm:main_intro} it follows that even in this case, at least the weaker property holds that $K_{X/T}$ is pseudo-effective. In fact, a similar phenomenon was known earlier: In \cite{Patakfalvi_Semi_positivity_in_positive_characteristics}, it was shown that $K_{X/T}$ is nef as soon as it is $f$-nef and the fibers have mild singularities. The novelty of \autoref{thm:main_intro} is to weaken these assumptions to the almost most general case, at the price of also weakening the conclusion from being nef to being pseudo-effective. In fact, the only possible further generalization  of \autoref{thm:main_intro} would be to remove the geometrically integral assumption, which we leave as an open question.

As usual for  results as above, \autoref{thm:main_intro} implies the following subadditivity of Kodaira dimension-type result, where we refer to the first paragraph of \autoref{ssec:proof} for the definition of the canonical divisor of $X_{\overline{\eta}}$ and for the fact that $K_{X_{\overline{\eta}}}= K_X|_{X_{\overline{\eta}}}$.

\begin{corollary}
\label{cor:subadditivity}
If\, $f \colon X \to T$ is a surjective morphism between smooth projective varieties over $k$  such that $T$ is of general type and the  geometric generic fiber $X_{\overline{\eta}}$ is integral, non-uniruled with $K_{X_{\overline{\eta}}}$ big, then
\begin{equation*}
\kappa(X) \geq \kappa\left( K_{X_{\overline{\eta}}} \right) + \kappa(T). 
\end{equation*}

\end{corollary}

The hardest in the proof of \autoref{thm:main_intro} is to construct a smooth non-uniruled finite cover of $T$. Being in the situation of $\characteristic p>0$ renders this hard in two aspects: 
\begin{itemize}
\item It is hard to show that a variety is not uniruled.
\item We need to use a construction that gives smoothness directly, as resolution of singularities is not available. 
\end{itemize}
So, next we state the by-product statements we obtained while constructing this non-uniruled cover. The first one is a simple non-uniruledness condition, in terms of coherent cohomology. The author is in fact not aware of earlier such conditions in the literature that work for arbitrary varieties. However, before the statement, we need to recall the notion of Frobenius-semi-stable part. 

If $X$ is a projective variety over $k$, then the absolute Frobenius morphism $F \colon X  \to X$ induces a homomorphism $F^* \colon H^i(X, \sO_X) \to H^i(X, \sO_X)$. This is referred to as the Frobenius action on $H^i(X, \sO_X)$. The \emph{semi-stable} part of $H^i(X, \sO_X)$ with respect to this action can be defined multiple ways. It is  given both by the $\bF_p$-linear subspace where $F^*$ acts by the identity and also by the image of a  high-enough iteration of $F^*$ (see, \textit{e.g.}, \cite[Lemma 3.3]{Chambert_Loir_Cohomologie_cristalline_un_survol}):
\begin{equation*}
H^i(X, \sO_X)^{\sstab}:= \left( H^i(X, \sO_X)^{F^* = \Id} \right) \otimes_{\bF_p} k = \left(F^*\right)^e \left( H^i(X, \sO_X) \right) \quad \textrm{for } e \gg 0.
\end{equation*}

\begin{theorem}[Smooth case of \autoref{prop:Witt_cohom_non_zero}]
\label{thm:non_uniruled}
If for a smooth projective variety $X$ over $k$ of dimension $n>0$, the inequality
\begin{equation*}
\dim_k H^{n-1}(X, \sO_X)^{\sstab} < \dim_k H^n (X, \sO_X)^{\sstab} \end{equation*}
holds, then  $X$ is not uniruled.
\end{theorem}

The construction of the smooth non-uniruled cover of $T$ is then given by the following corollary of \autoref{thm:non_uniruled}. 

\begin{corollary}[Smooth case of \autoref{thm:non_uniruled_cyclic_cover}]
\label{cor:cyclic_cover_intro}
Let $X$ be a smooth projective  variety over $k$, and let  $\sH$ be an ample line bundle on $X$. Then there exists an integer $s >0$ with the following property: For every integer $ p \nmid d >0$ and for every general $D \in |\sH^{sd}|$, the  corresponding degree $d$ cyclic cover
\begin{equation*}
Y:= \Spec_X
\expl{\left( \bigoplus_{i=0}^{d-1} \sH^{-si} \right)}{the algebra structure is given by $\sH^{-sd} \cong \sO_X(-D) \hookrightarrow \sO_X$}
\end{equation*}
 is not uniruled for $s \gg 0$. 

\end{corollary}

Lastly, we mention a direct consequence of \autoref{thm:non_uniruled} to mixed characteristic. The starting point is the fact that uniruledness is not a constructible property in mixed-characteristic families. In fact, \cite[Theorem IV.1.8.1]{Kollar_Rational_curves_on_algebraic_varieties} states that the locus in flat families where fibers are uniruled is a countable union of closed subvarieties. And, over mixed-characteristic bases that are of finite type over $\bZ$, this is the best on can say. For example, consider $X:=V(x^4+y^4 + z^4 +v^4 =0) \subseteq \bP^3_{\bZ[1/2]}$. By \cite[Theorem III]{Shioda_Katsura_On_Fermat_varieties}, $X_p$ is unirational and hence also uniruled whenever $p \equiv -1 (4)$. On the other hand,  if $p \equiv 1 (4)$, then $X_p$ is globally $F$-split, as the coefficient of $xyzv^{p-1}$ in $(x^4 + y^4 + z^4 + v^4)^{p-1}$ is non-zero. Equivalently, $X_p$ is weakly ordinary, and then for example \cite[Theorem 1.1]{Patakfalvi_Zdanowicz_Ordinary_varieties_with_trivial_canonical_bundle_are_not_uniruled}
implies that $X_p$ in this case is not uniruled. Hence, in this case both the uniruled and the non-uniruled loci are infinite, and they both have density $\frac{1}{2}$. 

Summarizing, the above example shows that the constructibility of neither the uniruled nor the non-uniruled locus holds. And in fact, both loci can be not only infinite, but even of high density. On the other hand, the author is not aware of results pertaining to general varieties claiming that these infinite behaviors are not only possible, but they happen whenever certain criteria are satisfied. This seems to be an extremely hard problem, especially if one would also say something about densities. Nevertheless, \autoref{thm:non_uniruled} implies a statement of this type assuming the weak ordinarity conjecture. 

The \emph{weak ordinarity conjecture}, see \cite[Conjecture~1.1]{Mustata_Srinivas_Ordinary_varieties_and}, states that if $X_S \to S$ is a smooth, projective family over an integral, mixed-characteristic base of finite type over $\Spec (\bZ)$ with generic point $\eta$, then the set 
\begin{equation*}
\left\{  \  s \in S \textrm{ a closed point } \left| \ \dim_{k(s)} H^i\left(X_s, \sO_{X_s}\right)^{\sstab} =  \dim_{k(\eta)} H^i\left( X_{\eta}, \sO_{X_{\eta}} \right) \  \right. \right\}
\end{equation*}
is dense.

\begin{corollary}
Let $X$ be a smooth projective variety over a field $k_0$ of characteristic zero such that $\dim_{k_0} H^{\dim X}(X, \sO_X)> \dim_{k_0} H^{\dim X -1} ( X, \sO_X)$, and let $X_S \to S$ be a model of $X$ over an integral, mixed-characteristic base of finite type over $\Spec (\bZ)$. 
Then, under the weakly ordinarity conjecture, the following set is dense in $S$:
\begin{equation*}
\left\{  \  s \in S \textrm{ a closed point } \left| \ X_s \textrm{ is not uniruled} \  \right. \right\}.
\end{equation*}

\end{corollary}

\subsection{The structure of the article and the outline of the proof}

The main idea of the proof of \autoref{thm:main_intro}, or rather of its singular version \autoref{thm:main}, is relatively straightforward; the proof is given in \autoref{sec:proof}. We use a bend-and-break argument together with a base-change to a non-uniruled cover of the base, a few iterated Frobenius base-changes,  and the fact that by now it is known in any characteristic that the pseudo-effective cone is the dual of the cone of movable curves. We refer to the proof of \autoref{thm:main} for the details, and here we only explain the main technical obstacle, which leads to the majority of the work done in the article: It is essential  that during the above-mentioned base-changes, the total space of the fibration stays integral and its singularities stay local complete intersections, so that bend-and-break applies. The only way we are able to guarantee this is if the base-changes are induced by finite flat covers of the base by smooth varieties. 

Hence, the majority of article, that is, \autoref{sec:non_uniruled_cover}, is about showing the existence of a finite flat smooth non-uniruled cover for any smooth projective variety $X$ of dimension $n$. This is done by first showing in \autoref{prop:Witt_cohom_non_zero},  which can be found in \autoref{sec:Witt_non_vanishing},   that if  the inequality 
\begin{equation}
\label{eq:outline_sstab}
H^{n-1}\left(X,\sO_X\right)^{\sstab}< H^n\left(X, \sO_X\right)^{\sstab}
\end{equation}
holds, then $H^n\left(X, W \sO_{X, \bQ}\right)\neq 0$. 
It has been known that this non-vanishing implies non-uniruledness in the case of $W \sO$-rational singularities; see \cite{Esnault_Varieties_over_a_finite_field_with_trivial_Chow_group_of_0_cycles_have_a_rational_point} and  
\cite[Proposition 4.6]{Patakfalvi_Zdanowicz_Ordinary_varieties_with_trivial_canonical_bundle_are_not_uniruled}. 
So, let us focus on how one shows the non-vanishing. By the definition of $H^n(X, W \sO_{X, \bQ})$, it is equivalent to finding an element $x \in H^n(X, W \sO_{X})$ such that $p^i x \neq 0$ for every integer $i>0$.
The main idea to achieve this is to show that inequality \autoref{eq:outline_sstab} implies that the length of the semi-stable part $H^n(X,W_i \sO_X)^{\sstab}$ is a strictly monotone function of $i$. Based on this realization, we show in the proof of \autoref{prop:Witt_cohom_non_zero} that there exist another strictly monotone sequence $j_i$ and compatible elements $x_{j_i} \in H^n\left(X,W_{j_i} \sO_X\right)^{\sstab}$ such that $p^i x_{j_i} \neq 0$. The key here is to realize that $p^i = V^i F^i$, so as $i \cdot \dim H^{n-1}(X, \sO_X)$ gives an upper bound for the kernel of $V^i$ and as $F$ acts by bijection on  $H^n\left(X,W_{j_i} \sO_X\right)^{\sstab}$, we have plenty of elements $y \in H^n\left(X,W_{j_i} \sO_X\right)^{\sstab}$ with $p^i y \neq 0$ as soon as we choose $j_i$ to be large enough. We refer for the finer details to the proof of \autoref{prop:Witt_cohom_non_zero}. Instead, we give two more general remarks:
\begin{enumerate}
\item We think it is essential that  in the above argument, we allow $j_i$ to be much larger than $i$, that is, one cannot  always find an $x=(x_i) \in H^n(X, W \sO_X) = \varprojlim H^n (X, W_i \sO_X)$ such that $p^i x_i \neq 0$. Unfortunately, giving a precise example to such behavior is very hard, as it needs an example of a variety with Bockstein operators that  either are injective or at least have very small kernel. So, we leave this as an open question. Nevertheless, we cannot exclude the existence of such varieties, and hence we are forced to allow $j_i$ to be much larger than $i$ in the proof of \autoref{prop:Witt_cohom_non_zero}.
\item The above argument needs a setup of some category into which $H^{j}(X, W_i \sO_X)$ fits,  which takes into account the Frobenius actions and hence using which one can talk about semi-stable submodules. This is a situation that resembles that of $F$-crystals, but instead of free $W(k)$ modules, we consider finite-length $W(k)$-modules. As we did not find a reference for this setting, we worked out the details in \autoref{sec:category}.
\end{enumerate}
Finally, we have to show that a cyclic cover $Y$ of $X$ as in \autoref{cor:cyclic_cover_intro} (or as in the singular version in \autoref{thm:non_uniruled_cyclic_cover}) satisfies condition \autoref{eq:outline_sstab}. As for such covers $H^{n-1}(Y, \sO_Y)$ is bounded, this boils down to showing that $\dim_k H^n(Y, \sO_Y)^{\sstab}$ grows indefinitely for a general $D$. This then boils down to showing that the semi-stable subspace $H^n(X, \sH^{-s})^{\sstab, D}$ with respect to the following Frobenius  action for general $D$ grows indefinitely as we increase $s$ (here $e >0$ is an integer such that $d | p^e -1$):
\begin{equation*}
\xymatrix{
\sH^{-s} \ar[r] & \sH^{-s} \otimes F^{e,*} \sO_X \cong F^e_* \sH^{-sp^e} \ar[rrr]^(0.6){\cdot \frac{p^e-1}{d}D} & & &  F^e_* \sH^{-s}. 
}
\end{equation*}
It is not hard to show this for a specific choice of $D$ as soon as $\sH^s$ is ample enough, using the local description of the Frobenius trace; see the proof of \autoref{thm:non_uniruled_cyclic_cover} in \autoref{sec:non_vanishing_Frob_action}. Then, we show  in \autoref{prop:non_zero_semi_stable_part}, which can be found in \autoref{sec:deformation_Frobenius_action}, that the statement deforms to the Frobenius action given by a general $D$.

\subsection*{Acknowledgements}

The author is particularly grateful to Maciej Zdanowicz, from whom he learned a big part of the Witt-cohomology techniques used in \autoref{sec:non_uniruled_cover} while working on their joint paper \cite{Patakfalvi_Zdanowicz_Ordinary_varieties_with_trivial_canonical_bundle_are_not_uniruled}, and with whom he had multiple extremely useful conversations about the article.
The author would like to thank in general the members of his group who were present at the group meetings when discussing the paper.

\section{Notation}

We fix an algebraically closed \emph{base-field} $k$ of characteristic $p>0$. In the present article, \emph{variety} means a quasi-projective, integral scheme over $k$. 

We use Witt-sheaves and Witt-cohomology in \autoref{sec:category} and \autoref{sec:Witt_non_vanishing}. For the notation concerning this, we refer to \cite[Section 2.5]{Gongyo_Nakamura_Tanaka_Rational_points_on_log_Fano_threefolds_over_a_finite_field}.

Throughout the article, the \emph{generic fiber} of a morphism $f \colon X\to T$ to an integral scheme is the fiber $X_\eta$ over the generic point $\eta = \Spec( K(T))$ of $T$. This is not to be confused with the general fiber of $f$, by which we mean the closed fibers over a non-empty open set of the base, assuming that $T$ is of finite type over $k$. In fact, the generic fiber typically behaves quite differently than the general fiber. On the other hand, the geometric generic fiber $X_{\overline{\eta}}$ usually has a singularity behavior similar to that of the general fiber (see, \textit{e.g.}, \cite[Proposition 2.1]{Patakfalvi_Waldron_Singularities_of_General_Fibers_and_the_LMMP}).

\subsection{Local complete intersection singularities}
\label{sec:local_compl_int}
We use the definition of \cite{stacks-project} for local complete intersection singularities. The definition of when a morphism $f \colon X \to Y$ is local complete intersection is given in \cite[Definition~\href{https://stacks.math.columbia.edu/tag/069F}{069F}]{stacks-project}, using Koszul-regularity. However, for locally Noetherian schemes, this agrees with the usual definition that it factors as 
\begin{equation*}
\xymatrix{
X \ar@/^1.5pc/[rr]^f \ar[r]_g & Z \ar[r]^h & Y\rlap{,}
}
\end{equation*}
where $h$ is smooth and $g$ is a closed embedding defined locally by a regular sequence; see 
\cite[Lemma~\href{https://stacks.math.columbia.edu/tag/063L}{063L} and Definitions~\href{https://stacks.math.columbia.edu/tag/063J}{063J}, \href{https://stacks.math.columbia.edu/tag/063D}{063D}, and \href{https://stacks.math.columbia.edu/tag/00LF}{00LF}]{stacks-project}.

If we specialize the above definition to $Y = \Spec (k)$, then we obtain the notion of local complete intersection singularities. That is, $X$ has local complete intersection singularities if locally around each $x \in X$, $X$ is a closed subscheme of a smooth variety $Z$, where the ideal of the embedding is generated by a regular sequence $f_i \in \sO_{X,x}$ $(i=1,\dots,r)$.  Equivalently, one can require that $\dim_x X = \dim_x Z -r$; see \cite[Theorem 17.4]{Matsumura_Commuatative_ring_theory}.\looseness=-1

\begin{proposition}
\label{lem:complete_intersection_singularities}\leavevmode
\begin{enumerate}
\item \label{itm:complete_intersection_singularities:relative_absolute} If $f \colon X \to Y$ is a complete intersection morphism and $Y$ has complete intersection singularities, then $X$ also has  complete intersection singularities.
\item \label{itm:complete_intersection_singularities:regular_over_regular} If $f \colon X \to Y$ is a morphism of finite type between regular Noetherian schemes, then $f$ is a complete intersection morphism.
\end{enumerate}
\end{proposition}

\begin{proof}
\eqref{itm:complete_intersection_singularities:relative_absolute}~
This is shown in \cite[Lemma~\href{https://stacks.math.columbia.edu/tag/069J}{069J}]{stacks-project}.

\eqref{itm:complete_intersection_singularities:regular_over_regular}~
  By the finite-type assumption, locally $f$ can be factored as 
\begin{equation*}
\xymatrix{
X \ar@/^1.5pc/[rr]^f \ar[r]_g & \bA^n_Y \ar[r]^h & Y\rlap{,}
}
\end{equation*}
where $g$ is a closed embedding. Applying \cite[Lemma~\href{https://stacks.math.columbia.edu/tag/0E9J}{0E9J}]{stacks-project} yields that $g$ locally is given by an ideal generated by a regular sequence. 
\end{proof}

We also note that complete intersection singularities are Cohen--Macaulay. So, if  $X$ has complete intersection singularities, then $X$ being reduced is decided at the generic points of $X$.

\section{A flat non-uniruled smooth cover}
\label{sec:non_uniruled_cover}

\subsection{Category of \texorpdfstring{$\boldsymbol{W(k)_{\sigma}}$}{W(k)\_sigma}-modules of finite \texorpdfstring{$\boldsymbol{W(k)}$}{W(k)}-length.}
\label{sec:category}

For the convenience of the reader, we present this material, which might be well known to the experts. The reason is that we need an $F$-module-like theory that applies  to cohomology groups of the form $H^i\left(X, W_j \sO_X\right)$, where $X$ is projective.

Before we state the definition of our main objects, let us also explain  two notational decisions:
\begin{itemize}
\item As usual in the theory of $F$-crystals, we also denote the Frobenius morphism on $W(k)$ by $\sigma \colon W(k) \to W(k)$ to avoid mixing it up with the Frobenius action on our modules. 

\item To avoid the clash of notation with the theory of $F$-crystals, which concerns free $W(k)$-modules, as explained above, we call our modules $W(k)_{\sigma}$-modules. 
\end{itemize}
Just for the definition of $W(k)_{\sigma}$-modules, we do not need to restrict to the case of finite $\length_{W(k)}$. We will impose this additional condition only later where it is necessary.

\begin{definition}
\label{def:Wk_sigma_module}
A \emph{$W(k)_{\sigma}$-module} is a pair $(M,F)$ such that  $M$ is a $W(k)$-module and $F \colon M \to M$ is an additive homomorphism such that
\begin{equation}
\label{eq:Wk_sigma_module}
 \forall m \in M, \forall r \in W(k) \ : \ F(r m) = \sigma(r)  F(m).
\end{equation}
A \emph{$W(k)_{\sigma}$-submodule} of a $W(k)_{\sigma}$-module $M$ is a $W(k)$-submodule $N \subseteq M$ such that $F(N) \subseteq N$.

\end{definition}

\begin{lemma}
\label{lem:sigma_invertible}
The action of $\sigma$ is invertible on $W(k) $.
\end{lemma}

\begin{proof}
Using the presentation of \cite[Section II.6]{Serre_Local_fields} or \cite[Section 2.5]{Gongyo_Nakamura_Tanaka_Rational_points_on_log_Fano_threefolds_over_a_finite_field}, the elements of $W(k)$ can be thought of as vectors $(a_0, a_1, \dots) \in k^{ \bN}$, and the map $\sigma$ is given by $(a_0, a_1, \dots) \mapsto \left(a_0^p, a_1^p,\dots\right)$. As $k$ is perfect, this is bijective.
\end{proof}

\begin{notation}
\label{notation:Witt_cohomologies}
Let $X$ be a projective variety over $k$ of dimension $n$. 
We have ring homomorphisms  
\begin{equation*}
R \colon W_{j+1}\sO_X \to W_j \sO_X,\quad 
 V \colon  W_j \sO_X \to W_{j+1} \sO_X,\quad
p \colon W_j \sO_X \to W_j \sO_X,\quad\text{and}\quad
F \colon W_j \sO_X \to W_j \sO_X.
\end{equation*} 
Using the notation of \cite[Section II.6]{Serre_Local_fields} or \cite[Section 2.5]{Gongyo_Nakamura_Tanaka_Rational_points_on_log_Fano_threefolds_over_a_finite_field}, these homomorphisms are given by
\begin{equation*}
R(a_0,\dots,a_j) = (a_0,\dots, a_{j-1}), \quad V(a_0,\dots, a_{j-1}) = (0,a_0,\dots, a_{j-1}),
\end{equation*}
\begin{equation*}
 F(a_0,\dots, a_{j-1})  = \left(a_0^p,\dots, a_{j-1}^p \right)
  , \quad \textrm{and} \quad
 p (a_0,\dots,a_{j-1}) = \left(0,a_0^p, \dots, a_{j-2}^p\right).
\end{equation*}
In particular, we have the relations
\begin{equation*}
p = VFR, \quad
F  V =  V  F, \quad p  F = F  p, \quad \textrm{ and }  \quad V  p = p V.
\end{equation*}
This then induces homomorphisms and also the respective relation after applying $H^i( X, \_ )$. By abuse of notation, we also denote  these induced homomorphisms by $R$, $V$, $p$, and $F$.
\end{notation}

\begin{remark}
\label{rem:Witt_cohomology_identities}
Using \autoref{notation:Witt_cohomologies}, the following properties will be important for us. Apart from the first one, these properties hold first on the ring level, and then consequently also after applying $H^i(X, \_)$ by functoriality. 

So, the properties are as follows, where $r \in W(k)$ and $m \in H^i\left(W_j \sO_X\right)$ are arbitrary:
\begin{enumerate}
\item $F$ and $R$ are ring homomorphism before applying $H^i(X, \_)$.
\item $F$, $V$, $R$, and $p$ are additive both before and after applying $H^i(X, \_)$.
\item $F(r \cdot m) = \sigma(r) \cdot m$.
\item $V(r \cdot m) = \sigma^{-1}(r) \cdot V(m)$, where  $\sigma \colon W(k) \to W(k)$ is bijective by \autoref{lem:sigma_invertible}. 
\item $p ( r \cdot m) = r \cdot (p m)$.
\end{enumerate}
\end{remark}

\begin{lemma}
\label{lem:Witt_coho_sigma_Wk_module}
If\, $X$ is a projective variety over $k$ and $i \geq 0$ and $j \geq 1$ are integers, then $H^i \left(X, W_j \sO_X\right)$ is a $W(k)_{\sigma}$-module  with $F \colon  H^i \left(X, W_j \sO_X\right) \to H^i \left(X, W_j \sO_X\right)$ being the structure homomorphism. 
\end{lemma}

\begin{proof}
This follows from the identities of \autoref{rem:Witt_cohomology_identities}. 
\end{proof}

Additionally, we want to work for different compositions of the maps $F$, $V$, $p$, $R$ and the Bockstein operators. These are not always $W(k)$-linear. This is the motivation for the notion of generalized $W(k)_\sigma$-module homomorphism, defined below. 

\begin{definition}
\label{def:Wk_sigma_homomorphism}
Let $M$ and $N$ be $W(k)_{\sigma}$-modules, and let $\alpha \colon M \to N$ be a map of sets. Then, 
\begin{enumerate}
\item 

 $\alpha$ is an \emph{additive $F$-homomorphism} if it is an additive homomorphism such that the following diagram commutes:
\begin{equation}
\label{eq:Wk_sigma_homomorphism:F_homomorphism}
\xymatrix{
M \ar[d]_F \ar[r]^-{\alpha} & N \ar[d]^F \\
M \ar[r]_{\alpha} & N\rlap{;}  
}
\end{equation}
\item $\alpha$ is a \emph{$W(k)_{\sigma}$-homomorphism} if it is both an additive $F$-homomorphism and a $W(k)$-module homomorphism; 
\item
\label{itm:Wk_sigma_homomorphism:generalized}
 $\alpha$ is a \emph{generalized $W(k)_{\sigma}$-homomorphism} if it is an additive $F$-homomorphism and there is an $i \in \bZ$ such that for every $r \in W(k)$, the following diagram commutes:
\begin{equation*}
\xymatrix{
M \ar[r]^-{\alpha} \ar[d]_{x \mapsto r \cdot x} & N \ar[d]^{x \mapsto \sigma^{i}(r) \cdot x} \\
M \ar[r]_{\alpha} & N\rlap{.} 
}
\end{equation*}
The integer $i$ is called the \emph{index of $\alpha$}.
\end{enumerate}

\end{definition}

\begin{example}
\label{ex:F_V_R_p_generalized}
According to \autoref{rem:Witt_cohomology_identities}, any composition of the maps $F$, $V$, $R$, and $p$ on $H^i\left(X, W_j \sO_X\right)$ is a generalized $W(k)_{\sigma}$-module homomorphism. 
\end{example}

We draw attention to the fact that in the next proposition, some statements are about $W(k)$-modules and others are about $W(k)_{\sigma}$-modules. 

\begin{lemma}
\label{lem:submodules}
Let $M$ and $N$ be $W(k)_{\sigma}$-modules, and let $\alpha \colon M \to N$ be a generalized $W(k)_{\sigma}$-module homomorphism.
\begin{enumerate}
\item \label{itm:submodules:image_of_submodules} If\, $L \subseteq M$ is a $W(k)$-submodule of\, $M$, then $\alpha(L)$ is a $W(k)$-submodule of\, $N$. 
\item \label{itm:submodules:preimage_of_submodules} If\, $L \subseteq N$ is  a $W(k)$-submodule of\, $N$, then $\alpha^{-1}(L)$ is a $W(k)$-submodule of\, $M$.
\item \label{itm:submodules:kernel} $\ker \alpha$ is a $W(k)_{\sigma}$-submodule of\, $M$.
\item \label{itm:submodules:image} $\im \alpha$ is a $W(k)_{\sigma}$-submodule of\, $N$.
\item \label{itm:submodules:coker} $\coker \alpha$ inherits a natural $W(k)_{\sigma}$-module structure from $N$. 
\item \label{itm:submodules:length} $\length_{W(k)} M - \length_{W(k)} \ker \alpha = \length_{W(k)} \im\alpha$. $($This is not obvious because according to \autoref{def:Wk_sigma_homomorphism}, $\alpha$ need not be a $W(k)$-homomorphism.$)$
\end{enumerate}
\end{lemma}

\begin{proof}
Let $i$ be the index of $\alpha$; see point \autoref{itm:Wk_sigma_homomorphism:generalized} of \autoref{def:Wk_sigma_homomorphism}.

\eqref{itm:submodules:image_of_submodules}~
As $\alpha$  is additive, $\alpha(L)$ is an additive subgroup of $N$.  To see that it is in fact a $W(k)$-submodule, 
choose $m \in L$ and $r \in W(k)$. Then $r \cdot \alpha(m) = \alpha\left(\sigma^{-i}(r) \cdot m\right) \in \alpha( L)$.

\eqref{itm:submodules:preimage_of_submodules}~
Similarly, choose $m \in \alpha^{-1}(L)$ and $r \in W(k)$. Then $\alpha(r \cdot m) = \sigma^i(r) \cdot \alpha(m) \in L$, and hence $r \cdot m \in \alpha^{-1}(L)$.

\eqref{itm:submodules:kernel}~
$\im \alpha= \alpha(M)$ is a $W(k)$-submodule by point \autoref{itm:submodules:image_of_submodules}. So, we only have to show that $F(\im \alpha) \subseteq \im \alpha$. This follows from the commutative diagram \autoref{eq:Wk_sigma_homomorphism:F_homomorphism}. 

\eqref{itm:submodules:image}~
Similarly, $\ker \alpha = \alpha^{-1}(0)$ is a $W(k)$-submodule by point \autoref{itm:submodules:preimage_of_submodules}. So, we only have to show that $F(\ker \alpha) \subseteq \ker \alpha$, which again follows from \autoref{eq:Wk_sigma_homomorphism:F_homomorphism}.

\eqref{itm:submodules:coker}~
By point \autoref{itm:submodules:image}, $\im \alpha$ is a $W(k)$-submodule of $M$, and hence $\coker \alpha$ inherits a natural $W(k)$-module structure. Additionally, $F$ also descends to $\coker \alpha$ as $F (\im \alpha) \subseteq \alpha$  holds  by point \autoref{itm:submodules:image} as well.

\eqref{itm:submodules:length}~
Note that $\alpha$ induces an additive bijection $\widetilde{\alpha} \colon \factor{M}{\ker \alpha}  \to \im \alpha$. Using that $\alpha$ is a generalized $W(k)_{\sigma}$-homomorphism, we see that so is $\widetilde{\alpha}$. The only reason we are not ready is that if the index is not zero, then $\alpha$ might not be an actual $W(k)$-module homomorphism. However, by points \autoref{itm:submodules:image_of_submodules} and \autoref{itm:submodules:preimage_of_submodules}, the chains of submodules of $\factor{M}{\ker \alpha}$ and $\im \alpha$ correspond to each other via $\widetilde{\alpha}$. This concludes our proof. 
\end{proof}

\begin{lemma}
\label{lem:finite_lenght}
For every integer $j >0$, 
$\length_{W(k)} H^i\left(X, W_j \sO_X\right) $ is finite. 
\end{lemma}

\begin{proof}
We give a proof by induction on $j$. If $j=1$, then $W_j \sO_X = \sO_X$, and the statement follows straight from the projectivity of $X$. For $j>1$, consider the exact sequence
\begin{equation}
\label{eq:finite_lenght:exact}
\xymatrix{
 H^i(X,\sO_X) \ar[r]^-V & H^i\left(X,W_j \sO_X\right) \ar[r]^-R & H^i\left(X,W_{j-1} \sO_X\right).
}
\end{equation}
According to \autoref{lem:Witt_coho_sigma_Wk_module} and \autoref{ex:F_V_R_p_generalized}, the modules and the maps of \autoref{eq:finite_lenght:exact} are $W(k)_{\sigma}$-modules and generalized $W(k)_{\sigma}$-module homomorphisms. By the induction hypothesis, their lengths over $W(k)$ are finite. Then \autoref{lem:submodules} concludes our proof. 
\end{proof}

\begin{definition}
For a $W(k)_\sigma$-module $M$ with $\length_{W(k)} M$ finite, consider $F^e ( M) \subseteq  M$  for every integer $e>0$.  Applying \autoref{lem:submodules} with $\alpha= F^e$, this gives a descending chain of $W(k)_{\sigma}$-submodules of $M$. As $\length_{W(k)} M< \infty$, this chain stabilizes. Hence, we may define 
\begin{equation*}
M^{\sstab}:= F^e ( M) \quad\textrm{for } e \gg 0.
\end{equation*}
Note: By point \autoref{itm:submodules:image} of \autoref{lem:submodules}, $M^{\sstab}$ is a $W(k)_{\sigma}$-submodule.
\end{definition}

\begin{proposition}
\label{prop:category}
Let $\sC$ be the category of finite $W(k)$-length $W(k)_{\sigma}$-modules with arrows being the generalized $W(k)_{\sigma}$-module homomorphisms. 
\begin{enumerate}
\item \label{itm:category:abelian}  For any arrow $\alpha$ in $\sC$,  $\ker \alpha$, $\im \alpha$, and $\coker \alpha$ are also in $\sC$  $($here $\ker \alpha$, $\im \alpha$, and $\coker \alpha$ are taken as for additive groups, and then they are endowed with a $W(k)_{\sigma}$-module structure using \autoref{lem:submodules}\,$)$. 
\item \label{itm:category:exact} $M \mapsto M^{\sstab}$ is an exact functor $\sC \to \sC$.
\item \label{itm:category:length} $\length_{W(k)} (\_)$ is additive in exact sequences in $\sC$. 
\end{enumerate}
\end{proposition}

\begin{remark}
The category $\sC$ of \autoref{prop:category} has somewhat unusual properties too, which are not mentioned in \autoref{prop:category}. This is due to allowing generalized $W(k)_{\sigma}$-module homomorphisms of different indices into the category. The main issue is that it is not possible to add homomorphisms with different indices, or to construct homomorphisms to products induced by component homomorphisms of different indices. Therefore, $\sC$ is not abelian, and it does not have products. One can solve this by disallowing generalized $W(k)_{\sigma}$-homomorphisms and introducing instead a notion of twist of the objects (\textit{i.e.}, twisting the $W(k)_{\sigma}$-structure by an adequate power of $\sigma$). This way one can turn every commutative diagram in $\sC$ into a commutative diagram in this more restrictive category by adequately twisting the modules. To avoid writing twists in each diagram, we chose the first approach. 
\end{remark}

We prove \autoref{prop:category} after a few more lemmas.

\begin{lemma}
\label{lem:Wk_sigma_modules_main}
Let $M$ and $N$ be two finite $W(k)$-length $W(k)_{\sigma}$-modules, and let $\alpha \colon M \to N$ be an additive  $F$-homomorphism. Then 
\begin{enumerate}
\item \label{itm:Wk_sigma_modules_main:isomorphism} $F|_{M^{\sstab}}$ is an isomorphism; 
\item \label{itm:Wk_sigma_modules_main:containment} $\alpha ( M^{\sstab})  \subseteq  N^{\sstab}$;  
\item \label{itm:Wk_sigma_modules_main:equals} if $\alpha$ is surjective, then $\alpha (M^{\sstab} ) = N^{\sstab}$. 
\end{enumerate}
\end{lemma}

\begin{proof}
\eqref{itm:Wk_sigma_modules_main:isomorphism}~
  Fix an $e>0$ such that $F^{e'}(M) = F^e(M)$ for every $e' \geq e$. Then $F|_{F^{e}(M)} \colon F^e(M) \to F^{e+1}(M)= F^e(M)$ is surjective. Now, point \autoref{itm:submodules:length} of \autoref{lem:submodules} shows that $\length = \ker \left(F|_{F^e(M)} \right) = 0$; that is, $F|_{F^{e}(M)}$ is bijective. 

\eqref{itm:Wk_sigma_modules_main:containment}~
  Diagram \autoref{eq:Wk_sigma_homomorphism:F_homomorphism} yields the following commutative diagram:
\begin{equation}
\label{eq:Wk_sigma_modules_main:commutes}
\xymatrix{
M \ar[d]_{F^e} \ar[r]^-{\alpha} & N \ar[d]^{F^e} \\
M \ar[r]_{\alpha} & N\rlap{.}  
}
\end{equation}
The statement of the present point then follows directly from diagram \autoref{eq:Wk_sigma_modules_main:commutes}.

\eqref{itm:Wk_sigma_modules_main:equals}~
This also follows from \autoref{eq:Wk_sigma_modules_main:commutes}, taking into account that $\alpha$ is surjective. 
\end{proof}

\begin{remark}
\label{rem:semi_stable_arrow}
Point \autoref{itm:Wk_sigma_modules_main:containment} of \autoref{lem:Wk_sigma_modules_main} implies that if $\alpha \colon M \to N$ is a generalized $W(k)_{\sigma}$-module homomorphism, then there is an induced generalized $W(k)_{\sigma}$-module homomorphism $\alpha^{\sstab} \colon M^{\sstab} \to N^{\sstab}$. 
\end{remark}

\begin{lemma}
\label{lem:semi_stable_exact}
Consider an exact sequence of\, $W(k)_{\sigma}$-modules of finite $W(k)$-length with arrows being generalized $W(k)_{\sigma}$-module homomorphisms:
\begin{equation*}
\xymatrix{
0 \ar[r] & M \ar[r]^-{\alpha} & N \ar[r]^-{\beta} & L \ar[r] & 0. 
}
\end{equation*}
Then 
\begin{equation*}
\xymatrix{
0 \ar[r] & M^{\sstab} \ar[r]^-{\alpha^{\sstab}} & N^{\sstab} \ar[r]^-{\beta^{\sstab}} & L^{\sstab} \ar[r] & 0 
}
\end{equation*}
is exact.
\end{lemma}

\begin{proof}
According to point \autoref{itm:Wk_sigma_modules_main:equals} of \autoref{lem:Wk_sigma_modules_main}, we only have to show that $\ker \beta^{\sstab} = \im \alpha^{\sstab}$. This is equivalent to showing that $\alpha(M^{\sstab}) = N^{\sstab} \cap \ker \beta$, which is further equivalent to $(\ker \beta)^{\sstab} = N^{\sstab} \cap \ker \beta$. We prove this last one. As $(\ker \beta)^{\sstab}$ is contained in both $N^{\sstab}$ and $\ker \beta$, we have $(\ker \beta)^{\sstab} \subseteq N^{\sstab} \cap \ker \beta$. So, we only have to show the opposite containment, for which it is enough to show that $F|_{N^{\sstab}  \cap \ker \beta}$ is bijective. 

We note at this point that as both $\ker \beta$ and $N^{\sstab}$ are $W(k)_{\sigma}$-submodules of $N$, so is $N^{\sstab} \cap \ker \beta$. In particular, $F|_{N^{\sstab}  \cap \ker \beta}$ is a generalized $W(k)_{\sigma}$-module endomorphism of $N^{\sstab} \cap \ker \beta$. Additionally, by point~\autoref{itm:Wk_sigma_modules_main:isomorphism} of \autoref{lem:Wk_sigma_modules_main}, $F|_{N^{\sstab}}$ is injective. So, we obtain that $F|_{N^{\sstab}  \cap \ker \beta}$ is an injective generalized $W(k)_{\sigma}$-module endomorphism of $N^{\sstab} \cap \ker \beta$. Point \autoref{itm:submodules:length} of \autoref{lem:submodules} then shows that this endomorphism in fact has to be surjective, and hence  bijective. 
\end{proof}

\begin{proof}[Proof of \autoref{prop:category}]
Point \autoref{itm:category:abelian} is shown in \autoref{lem:submodules}. Point \autoref{itm:category:exact} is shown in \autoref{rem:semi_stable_arrow} and \autoref{lem:semi_stable_exact}. Point \autoref{itm:category:length} is shown in point \autoref{itm:submodules:length} of \autoref{lem:submodules}.
\end{proof}

\subsection{Witt non-vanishing criterion}
\label{sec:Witt_non_vanishing}

\begin{theorem}
\label{prop:Witt_cohom_non_zero}
If for a projective variety $X$ over $k$ of dimension $n>0$, the inequality
\begin{equation}
\label{eq:Witt_cohom_non_zero:assumption}
\dim_k H^{n-1}(X, \sO_X)^{\sstab} < \dim_k H^n (X, \sO_X)^{\sstab}
\end{equation}
holds, then $H^n(X, W \sO_{X, \bQ} ) \neq 0$. 

In particular, if $X$ additionally is normal and it has $W \sO$-rational singularities $($e.g., $X$ is smooth$)$, then $X$ is not uniruled.
\end{theorem}

\begin{proof}
The addendum follows directly from \cite[Proposition~4.6]{Patakfalvi_Zdanowicz_Ordinary_varieties_with_trivial_canonical_bundle_are_not_uniruled}.
So, we only show the statement that $H^n\left(X, W\sO_{X,\bQ}\right) \neq 0$.

Throughout the rest of the proof, all our cohomology groups are in the category $\sC$ of \autoref{prop:category}, except a one-time mention of $H^n(X, W\sO_X)$. In particular, we will use the statements of \autoref{prop:category}, without each time explicitly indicating a reference to that proposition.

\begin{enumerate}[label=\fbox{Step \arabic*:}, ref=\arabic*,wide]
\setcounter{enumi}{-1}
\item\label{step0} 
{\it Initial setup.} 
Set 
\begin{equation*}
r:= \dim_k H^n (X, \sO_X)^{\sstab} = \length_{W(k)} H^n (X, \sO_X)^{\sstab}.
\end{equation*}
It is enough to exhibit 
\begin{equation*}
x = (x_j) \in H^n(X, W \sO_{X} )= \varprojlim H^n\left(X, W_j \sO_X\right)
\end{equation*}
 such that 
\begin{equation}
\label{eq:Witt_cohom_non_zero:goal}
 \forall i  \geq 1,   \ p^i x \neq 0 
 \quad
 \Longleftrightarrow
 \quad
 \forall i \geq  1, \  \exists j_i >0 \ :\  p^ix_{j_i} \neq 0.
\end{equation}
So, our goal is to exhibit $x_{j_i}$ as above satisfying the second equivalent condition of \autoref{eq:Witt_cohom_non_zero:goal}. We will do this by induction on $i$, and  we will  choose $x_{j_i}$ such that $x_{j_i} \in H^n \left(X, W_{j_i} \sO_X \right)^{\sstab}$.

\item\label{step1} 
{\it The semi-stable subspace grows indefinitely.}  Consider, for any integer $j \geq 0$, the exact sequence 
\begin{equation*}
\xymatrix{
H^{n-1}\left(X, W_j\sO_X\right) \ar[r]^-B & H^n(X, \sO_X) \ar[r]^-V & H^n\left(X,W_{j+1} \sO_X\right) \ar[r]^-R & H^n\left(X,W_j \sO_X\right) \ar[r] & 0. 
}
\end{equation*}
By taking the semi-stable subspace, we obtain another exact sequence:
\begin{equation}
\label{eq:Witt_cohom_non_zero:basic_coho}
\xymatrix{
H^{n-1}\left(X, W_j\sO_X\right)^{\sstab} \ar[r]^-{B^{\sstab}} & H^n(X, \sO_X)^{\sstab} \ar[r]^-{V^{\sstab}} & H^n\left(X,W_{j+1} \sO_X\right)^{\sstab} \ar[r]^-{R^{\sstab}} & H^n\left(X,W_j \sO_X\right)^{\sstab} \ar[r] & 0. 
}
\end{equation}
By taking $\length_{W(k)}(\_)$ and using \autoref{eq:Witt_cohom_non_zero:assumption},  we obtain that 
\begin{equation}
\label{eq:Witt_cohom_non_zero:grows}
\length_{W(k)} H^n\left(X,W_{j+1} \sO_X\right)^{\sstab} > \length_{W(k)} H^n\left(X,W_j \sO_X\right)^{\sstab} \Longrightarrow \length_{W(k)} H^n\left(X,W_j \sO_X\right)^{\sstab} \geq j .
\end{equation}

\item\label{step2} 
{\it $R^{\sstab} \colon H^n\left(X, W_{j+1} \sO_X\right)^{\sstab} \to H^n\left(X,W_j \sO_X\right)^{\sstab}$ is surjective.} This follows from $(\_)^{\sstab}$ being an exact functor and  $H^n\left(X, W_{j+1} \sO_X\right) \to H^n\left(X,W_j \sO_X\right)$ being surjective. 

\item\label{step3} 
  {\it Start of the induction.} As $H^n(X, \sO_X) \neq 0$ by \autoref{eq:Witt_cohom_non_zero:assumption}, we may set $j_0:=1$, and we may choose $x_{j_0} \in H^n(X, \sO_X)$ to be any non-zero element. 

\item\label{step4} 
  {\it Induction step, initial setup.} So, fix an integer $i>0$, and assume that $x_{j_{i-1}} \in H^n\left(X, W_{j_{i-1}} \sO_X\right)^{\sstab}$ is chosen. In particular, we have $p^{i-1} x_{j_{i-1}} \neq 0$. We have to choose $j_i> j_{i-1}$ and  $x_{j_i} \in H^n\left(X, W_{j_i} \sO_X\right)^{\sstab}$ such that $R^{j_i - j_{i-1}}\left(x_{j_i} \right)= x_{j_{i-1}}$ and $p^i x_{j_i} \neq 0$. 
 
Now consider,  for any integer $t>j_{i-1}$, the diagram
\begin{equation*}
\xymatrix{
0 \ar[r]  & W_{t} \sO_X \ar[r]^-{V^i} & W_{t+i} \sO_X  \ar[r]^-{R^{t}} & W_i \sO_X \ar[r] & 0. 
}
\end{equation*}
Taking cohomology and then the semi-stable subspace, we obtain 
\begin{equation}
\label{eq:Witt_cohom_non_zero:Verschiebung}
\xymatrix{
H^{n-1}\left(X, W_{i} \sO_X\right)^{\sstab}   \ar[r]^-{B^{\sstab}_{i,t}}  & H^n\left(X,W_{t} \sO_X\right)^{\sstab} \ar[r]^-{V^i} &  H^n\left(X,W_{t+i} \sO_X\right)^{\sstab}/  
}
\end{equation}
According to \autoref{eq:Witt_cohom_non_zero:grows}, we may choose a $t > j_{i-1}$ such that $Z:=\Ker \alpha \not\subseteq \im B^{\sstab}_{i,t}=:M$, where $\alpha$ is the homomorphism  $\left(R^{t-j_{i-1}}\right)^{\sstab} \colon H^n\left(X, W_t \sO_X\right)^{\sstab} \to H^n\left(X, W_{j_{i-1}} \sO_X\right)^{\sstab} $. Fix this value of $t$, and set $j_i:=t+i$. 

\item\label{step5} 
  {\it We claim that $ \alpha^{-1} \left( x_{j_{i-1}} \right) \not\subseteq M$.} Indeed, assume the opposite, that is, that $ \alpha^{-1} \left( x_{j_{i-1}} \right) \subseteq M$. By Step~\ref{step2}, there is a $ z \in \alpha^{-1} \left( x_{j_{i-1}} \right)$. Hence, we have $Z+z =\alpha^{-1} \left( x_{j_{i-1}} \right)$, and then $Z +z \subseteq M$. Using the fact that $M$ is additively closed, we then have  the implications
  \[z \in  M \Longrightarrow -z \in M \Longrightarrow -z + (z+ Z ) = Z \subseteq M.\] This contradicts the choice of $t$ made in Step~\ref{step4}. 

\item\label{step6} 
  {\it Conclusion of the induction step.}
By Step~\ref{step5}, we may choose a $z'$ in $\alpha^{-1} \left( x_{j_{i-1}} \right) \setminus M$. In particular, as $M$ is the image of the left-side map of the exact sequence in \autoref{eq:Witt_cohom_non_zero:Verschiebung}, we obtain that $V^i  (z') \neq 0$. Now choose  $x_{j_i}$ to be any element of $H^n\left(X, W_{j_i} \sO_X \right)^{\sstab}$ mapping to $z'$. By Step~\ref{step2}, this is possible. Additionally, as $V^i(z') \neq 0$, we have
\begin{equation*}\pushQED{\qed}
p^i x_{j_i} = F^i V^i R^i \left( x_{j_i} \right) = 
F^i V^i (z')
\expl{\neq}{$0 \neq V^i(z') \in H^n \left( X, W_{j_i} \sO_X \right)^{\sstab}$, and $F$ is bijective on $H^n \left( X, W_{j_i} \sO_X \right)^{\sstab}$}
 0\qedhere \popQED
\end{equation*}
\end{enumerate}\renewcommand{\qed}{}  
\end{proof}

\subsection{Deformation of (Frobenius) semi-stable subspaces}
\label{sec:deformation_Frobenius_action}

Recall the following way of defining different Frobenius actions on a fixed line bundle: Let $L$ be a line bundle on a projective scheme $X$ over $k$ of dimension $n$, let $p \nmid d$ be an integer, and let $D \in \left| L^d \right|$ be a divisor. Also fix  an integer $e>0$ such that $d | p^e -1$. Then, one may define a Frobenius action induced by $D$ on $L^{-1}$ given by the following composition:
\begin{equation}
\label{eq:Frobenius_action}
\xymatrix@C=60pt{
L^{-1} \ar[r] \ar@/^1.5pc/[rr]^{\eta_{L,D}} & L^{-1}  \otimes F_*^e \sO_X 
\expl{\cong}{projection formula} 
F_*^e F^{e,*} L^{-1} \cong F_*^e L^{-p^e}
\ar[r]_(0.7){\cdot \frac{p^e-1}{d} D} & 
F^e_* L^{-1}. 
}
\end{equation}
Applying $H^n(X,\_ )$ to the this composition, we obtain a $p^e$-linear action $\psi_{L,D}$ on $H^n\left(X,L^{-1}\right)$. We denote by  $H^n\left(X, L^{-1}\right)^{\sstab,D}$  the semi-stable part with respect to $\psi_{L,D}$, that is, the image of a high-enough iteration of $\psi_{L,D}$.
We suppressed the integer $e$ from the notation of the semi-stable part, as it is an elementary exercise to see that the action  is independent of the choice of $e$ up to passing to a divisible-enough iteration. 

Recall that a perfect point $y$ of a scheme $Y$ is a morphism $\Spec (L) \to Y$ such that $L$ is a perfect field. 

\begin{lemma}
\label{lem:subsheaf_dim_bound}
Let $\sE$ be a locally free sheaf of finite rank over a Noetherian integral scheme $Y$, and let $\sF \subseteq \sE$ be a coherent subsheaf. Then for every $y \in Y$, 
\begin{equation*}
\rk \sF \geq \dim_{k(y)} \im \left( \sF \otimes k(y) \to \sE \otimes k(y)\right).
\end{equation*}
\end{lemma}

\begin{proof}
Set $I:= \im \left( \sF \otimes k(y) \to \sE \otimes k(y)\right)$, and set $r:= \dim_{k(y)}I$. We are supposed to show that $\rk \sF \geq r$.
We have 
\begin{equation}
\label{eq:subsheaf_dim_bound:first}
 \dim_{k(y)} \left( k(y) \otimes \left( \factor{\sE_y}{\sF_y} \right) \right) 
 \expl{=}{right exactness of tensor product}
 \dim_{k(y)} \left(\sE \otimes k(y) \right)  -  \dim_{k(y)}  I  
 \expl{=}{$\sE$ is locally free} 
    \rk \sE - r. 
\end{equation}
Hence, if $\eta$ is the generic point of $Y$, then the following computation concludes our proof:
\begin{equation*}\pushQED{\qed}
\rk \sE - \rk \sF = \rk  \left(\factor{\sE_y}{\sF_y} \right) 
= 
\dim_{k(\eta)} \left( k(\eta) \otimes \left(\factor{\sE_y}{\sF_y} \right) \right)
\expl{\leq}{\cite[Exercise II.5.8]{Hartshorne_Algebraic_geometry}}
\dim_{k(y)} \left( k(y) \otimes \left(\factor{\sE_y}{\sF_y} \right) \right)
\expl{=}{\autoref{eq:subsheaf_dim_bound:first}}
\rk \sE - r. \qedhere \popQED
 \end{equation*}
\renewcommand{\qed}{}     
\end{proof}

\begin{proposition}
\label{lem:stabilization_genericity}
Consider the following situation: 
\begin{enumerate}
\item Let $f \colon X \to Y$ be  a projective, flat, Gorenstein morphism  between varieties over $k$ with geometrically integral fibers of dimension $n$; 
\item \label{itm:stabilization_genericity:constant} let $L$ be a line bundle on $X$ such that $\dim_{k(y)} H^n\left(X_y, L_y^{-1} \right)$ is a constant function of $y \in Y$; 
\item let $d>0$ be an integer such that $p \nmid d$; 
\item let $D$ be an effective divisor on $X$, not containing any fiber, such that $\sO_X(D) \cong L^s $;  and
\item \label{itm:stabilization_genericity:non_zero} suppose that for some perfect point $y_0 \in Y$, we have $l:=\dim_{k(y_0)} H^n\left(X_{y_0}, L_{y_0}^{-1}\right)^{\sstab, D_{y_0}} >0$.
\end{enumerate}
Then, there is a non-empty open set $U \subseteq Y$ such that $\dim_{k(y)} H^n\left(X_{y}, L_{y}^{-1}\right)^{\sstab, D_{y}} \geq l$ for every perfect point $y \in Y$.
\end{proposition}

\begin{proof}
The main technical difficulty in proving \autoref{lem:stabilization_genericity} is that one needs to work with a relative version of the Frobenius morphism. That is, one needs a morphism that restricts on each (perfect) fiber to the Frobenius morphism of the corresponding fiber. This morphism is called the $\supth{e}$ relative Frobenius morphism $F_{\rel}^e$ of $f$, and it exists only after an adequate iterated Frobenius base-change, as shown on the following diagram:
\begin{equation*}
\xymatrix@R=10pt@C=50pt{
X^e \ar[rdd]_{f^e} \ar[rd]|{F_{\rel}^e} \ar[rrd]^{F^e_X}\\
& X \times_Y Y^e  \ar[r] \ar[d]  & X \ar[d]^f \\
& Y^e \ar[r]_{F^e_Y} & Y\rlap{.}
}
\end{equation*}
  Additionally, the depth of this base-change depends on the considered iteration of the Frobenius action on the fibers. Keeping track of these base-changes is notationally somewhat burdensome. 

So, we have to work with Frobenius pullbacks of the base. Hence we adopt the following notation:
\begin{itemize}
\item  Choose an integer $e >0$ such that $d | p^e -1$.
\item Set $r:= \frac{p^e-1}{d}$. 
\item  Since the statement is local, we may assume that $Y$ is affine and regular, with $A= \Gamma(Y, \sO_Y)$. 
\item Set $\phi:= F^e$ and $A_i:=A^{1/p^{ie}}=F^{ie}_* A$. In fact, by abuse of notation, we will use $\phi$ for $F^e_S$, where $S$ is any of the schemes appearing in the proof.
\item Let $\xi$ be the natural morphism $A \to A_1$ sending $x$ to $x  = \left( x^{1/p^{ie}} \right)^{p^{ie}}$. 
\end{itemize}
Consider the following commutative diagram of relative Frobenii of $f$, where we used the commutative algebra notation on the right side and the algebraic geometry notation on the left side. In fact, as all considered schemes are finite inseparable over $X$, diagram \autoref{eq:stabilization_genericity:rel_Frobenius_structure_sheaf} contains  the pushforwards  to $X$ of the structure sheaves of the considered spaces, instead of the spaces themselves. Also note that for the whole proof, pushforward is understood to have higher priority in the order of operations than tensor product.
\begin{equation}
\label{eq:stabilization_genericity:rel_Frobenius_structure_sheaf}
\xymatrix@R=35pt{
\sO_X \otimes_A  A_s  
\ar[d]^-{\xi \otimes_{ A_1} A_s} 
\ar@{=}[r] & \sO_X \otimes_A A^{1/p^{se}} \ar[d] \\
 \phi_*\sO_X\otimes_{ A_1}  A_s 
=  \phi_* \left( \sO_X\otimes_A  A_{s-1} \right)
\ar[d]^-{\phi_* \xi \otimes_{A_2} A_s 
=\phi_* \left( \xi \otimes_{A_1} A_{s-1} \right)}  
\ar@{=}[r] & \sO_X^{1/p^e} \otimes_{A^{1/p^e}} A^{1/p^{se}} \ar[d] 
\\
\vdots
\ar[d]^-{\phi^{s-2}_* \xi \otimes_{A_{s-1}} A_s
=\phi^{s-2}_* \left( \xi \otimes_{A_1} A_2 \right)}  
& \vdots \ar[d]  \\
 \phi_*^{s-1} \sO_X \otimes_{ A_{s-1}}  A_s 
=  \phi_*^{s-1} \left( \sO_X \otimes_A  A_{1} \right)
 \ar[d]^-{\phi^{s-1}_* \xi}
\ar@{=}[r] & \sO_X^{1/p^{(s-1)e}} \otimes_{A^{1/p^{(s-1)e}}} A^{1/p^{se}} \ar[d] 
 \\
\phi^{s}_*  \sO_X \ar@{=}[r] & \sO_X^{1/p^{se}}\rlap{.}
}
\end{equation}
The notable feature of diagram \autoref{eq:stabilization_genericity:rel_Frobenius_structure_sheaf} is the following:
\begin{equation}
\label{eq:stabilization_genericity:Frobenii_of_fibers}
\parbox{400pt}{
For any perfect point $y \in Y$ (or equivalently a non-zero $k$-algebra homomorphism $A \to L=:k(y)$, where $L$ is a perfect field), by restricting \autoref{eq:stabilization_genericity:rel_Frobenius_structure_sheaf} to $y^{1/p^{se}}$ \big(or equivalently by applying $(\_) \otimes_{A^{1/p^{se}}} k(y)^{1/p^{se}}$\big), we obtain the iterated relative Frobenii of $X_y$, and then by the perfectness of $k(y)$, we may identify these morphisms with the iterated absolute Frobenii of $X_y$.
 }
 \end{equation}
Now tensor  the left side of \autoref{eq:stabilization_genericity:rel_Frobenius_structure_sheaf} by $L^{-1}$ over $A$. Using a considerable amount of projection formulas together with the fact that $\phi^* L^{-1} \cong L^{-p^e}$, we obtain the following commutative diagram, where we define  $\zeta\colon L^{-1} \otimes_A A_1 \to \phi_* L^{-p^e}$ by $\zeta:= \xi \otimes_A L^{-1}$: 
\begin{equation}
\label{eq:stabilization_genericity:rel_Frobenius_L_no_pushforward_no_D}
\xymatrix@R=35pt{
L^{-1} \otimes_A  A_s 
\ar[d]^-{\zeta \otimes_{ A_1} A_s} \\
 \phi_* L^{-p^e} \otimes_{ A_1}  A_s =  \phi_* \left( L^{-p^e} \otimes_A  A_{s-1} \right)
\ar[d]^-{\phi_* \zeta \otimes_{A_2} A_s 
=\phi_* \left( \zeta \otimes_{A_1} A_{s-1} \right)}  
\\
\vdots
\ar[d]^-{\phi^{s-2}_* \zeta \otimes_{A_{s-1}} A_s
=\phi^{s-2}_* \left( \zeta \otimes_{A_1} A_2 \right)}  
\\
 \phi_*^{s-1} L^{-p^{(s-1)e}} \otimes_{ A_{s-1}}  A_s =  \phi_*^{s-1} \left( L^{-p^{(s-1)e}} \otimes_A  A_{1} \right)
 \ar[d]^-{\phi^{s-1}_* \zeta} \\
\phi^{s}_*  L^{-p^{se}}\rlap{.}
}
\end{equation}
Now modify   \autoref{eq:stabilization_genericity:rel_Frobenius_L_no_pushforward_no_D} so that after each homomorphism, we apply multiplication by $rD$ (after also applying  adequate $\phi^i_*(\_)$ and $(\_) \otimes_{A_j} A_s$). This way we obtain the following diagram, where $\eta$ is the composition of $\zeta$ with multiplication by  $rD_{A_1}$:
\begin{equation}
\label{eq:stabilization_genericity:rel_Frobenius_L_no_pushforward}
\xymatrix@R=35pt{
L^{-1} \otimes_A  A_s 
\ar[d]^-{\eta \otimes_{ A_1} A_s} \\
 \phi_* L^{-1} \otimes_{ A_1}  A_s =  \phi_* \left( L^{-1} \otimes_A  A_{s-1} \right)
\ar[d]^-{\phi_* \eta \otimes_{A_2} A_s
=\phi_* \left( \eta \otimes_{A_1} A_{s-1} \right)}  
\\
\vdots
\ar[d]^-{\phi^{s-2}_* \eta \otimes_{A_{s-1}} A_s
=\phi^{s-2}_* \left( \eta \otimes_{A_1} A_2 \right)}  
\\
 \phi_*^{s-1} L^{-1} \otimes_{ A_{s-1}}  A_s =  \phi_*^{s-1} \left( L^{-1} \otimes_A  A_{1} \right)
 \ar[d]^-{\phi^{s-1}_* \eta} \\
\phi^{s}_*  L^{-1}\rlap{.}
}
\end{equation}
Using the notation of \autoref{eq:stabilization_genericity:Frobenii_of_fibers},  we obtain that 
\begin{equation}
\label{eq:stabilization_genericity:Frobenii_of_fibers_L}
\parbox{400pt}{
the restriction of the homomorphisms of \autoref{eq:stabilization_genericity:rel_Frobenius_L_no_pushforward} over $y$ can be identified with the iterations of $\eta_{L_y,D_y}$.
}
\end{equation}
Now apply  $R^nf_*(\_)$ to \autoref{eq:stabilization_genericity:rel_Frobenius_L_no_pushforward}. Note that as $n$ is the dimension of all fibers of $f$, in this situation $R^nf_*(\_)$ commutes with arbitrary base-change for coherent sheaves flat over $Y$.  Hence, if we define $\sE:= R^n f_* L^{-1}$, then we obtain, for every integer $0 \leq i \leq s$, 
\begin{multline}
\label{eq:stabilization_genericity:coho_and_base_change}
R^n f_* \left(\phi_*^i L^{-1} \otimes_{A_i} A_s \right)  
\cong
R^nf_* \phi_*^i \left( L^{-1} \otimes_A A_{s-i} \right) 
\explshift{-80pt}{\cong}{$f \circ \phi = \phi \circ f$ and $\phi$ is affine}
\phi^i_* R^n f_* \left(L^{-1} \otimes_A A_{s-i}\right)
\\ \explshift{-140pt}{\cong}{$R^n f_* (\_)$ commutes with arbitrary base-change for coherent sheaves flat over $Y$}
\phi^s_* \phi^{s-i,*}  R^n f_* L^{-1}
=
\phi^s_* \phi^{s-i,*} \sE.
\end{multline}
Similarly, by defining $\psi = R^n f_*  (\eta)$, we obtain, for every integer $0 \leq i \leq s-1$, 
\begin{equation}
\label{eq:stabilization_genericity:coho_and_base_change_map}
R^n f_* \phi^{i}_* \left( \eta \otimes_{A_1} A_{s-i} \right) = \phi^s_* \phi^{s-i-1,*}(R^n f_*(\eta)) = \phi^s_* \phi^{s-i-1,*}(\psi).
\end{equation}
 Combining \autoref{eq:stabilization_genericity:rel_Frobenius_L_no_pushforward}, \autoref{eq:stabilization_genericity:coho_and_base_change}, and \autoref{eq:stabilization_genericity:coho_and_base_change_map}, and  disregarding the $\phi^s_*$, we obtain the following  commutative diagram:
\begin{equation}
\label{eq:stabilization_genericity:rel_Frobenius_L}
\xymatrix@C=65pt{
\phi^{s,*} \sE \ar@/_1pc/[rrrr]_{\psi^s} \ar[r]^-{\phi^{s-1,*}(\psi)} & 
\phi^{s-1,*} \sE \ar[r]^-{\phi^{s-2,*}(\psi)} & \cdots
\ar[r]^-{\phi^*(\psi)} & \phi^* \sE  \ar[r]^-{\psi} &
\sE. 
}
\end{equation} 
Using the notation of \autoref{eq:stabilization_genericity:Frobenii_of_fibers} and \autoref{eq:stabilization_genericity:Frobenii_of_fibers_L}, as $R^nf_*( \_)$ commutes with arbitrary base-change, we see that 
\begin{equation}
\label{eq:stabilization_genericity:Frobenii_of_fibers_H_n_L}
\parbox{400pt}{ the restriction of the homomorphisms of \autoref{eq:stabilization_genericity:rel_Frobenius_L} over $y$ can be identified with the iterations $\psi_{L_y,D_y}$. 
}
\end{equation}
However, a warning should be given here: \autoref{eq:stabilization_genericity:Frobenii_of_fibers_H_n_L} does not mean that $(\im \psi_y ) \otimes k(y)  = \im \psi_{L_y, D_y}$, where $\psi_{L_y, D_y}$ is defined in \autoref{eq:Frobenius_action}. In fact,  we have 
\begin{equation}
\label{eq:stabilization_genericity:image}
\im \left( \psi_{L_y, D_y} \right)^s = \im \left(  \left(\im \psi^s \right)  \otimes k(y) \lra \sE  \otimes k(y) \cong H^n\left(X_y, L_y^{-1}\right) \right).
\end{equation}
Now note  that assumption \autoref{itm:stabilization_genericity:constant} implies that $\sE$ is locally free. 
Then consider  \autoref{eq:stabilization_genericity:image} for the special case $y=y_0$. By assumption \autoref{itm:stabilization_genericity:non_zero} and \autoref{lem:subsheaf_dim_bound}, we obtain that $\rk (\im \psi^s) \geq l$.  In particular, as $\rk (\im \psi^s)$ is a monotone decreasing function of $s$, there is an integer $t$ such that  $\rk (\im \psi^s)$ is the same positive number for every $s \geq t$.  Note that by assumption \autoref{itm:stabilization_genericity:non_zero}, this number is at least $l$.

We claim  the following disjointness of subsheaves  of $\phi^{t,*} \sE$:
\begin{equation}
\label{eq:stabilization_genericity:disjoint}
\left( \phi^{t,*} \im \psi^t \right) \cap \ker \psi^t =0.
\end{equation}
Indeed, if the intersection was not zero, then as $\sE$ is locally free, the intersection would have positive rank, and hence $\im \psi^{2t} = \psi^t\left(  \phi^{t,*} \im \psi^t \right)$ would have rank smaller than that of $\im \psi^t$. This is  impossible by the choice of $t$, showing \autoref{eq:stabilization_genericity:disjoint}.

Equation \autoref{eq:stabilization_genericity:disjoint} implies that $\psi_t|_{\phi^{t,*} \im \psi^t} \colon \phi^{t,*} \im \psi^t  \to \im \psi^t$ is an isomorphism. Hence, for any integer $j>0$, we have $\im \psi^{jt} = \im \psi^t$. Additionally, by shrinking $Y$ we may assume that both $\im \psi^t$  and $\factor{\sE}{\im \psi^t}$ are locally free. In particular, for all $y \in Y$, $\left( \im \psi^t \right) \otimes k(y) \to \sE \otimes k(y)$ is an injection. Then \autoref{eq:stabilization_genericity:image} shows that for every integer $j>0$, we have $\dim_{k(y)} \im \left( \psi_{L_y, D_y} \right)^{jt}  = \rk \left(\im \psi^t\right) \geq l$. This concludes our proof. 
\end{proof}

\subsection{Non-vanishing of a specific Frobenius action}
\label{sec:non_vanishing_Frob_action}

\begin{remark}
\label{rem:Frobenius_trace_locally}
In what follows, the following fact will be essential: If $X$ is any variety and $x \in X_{\reg}$ is a closed point and $t_1,\dots, t_n$ is a system of regular parameters at $x$, then the trace homomorphism $\Tr_{F^e}\colon F_*^e \omega_X \to \omega_X$ can be identified in the formal neighborhood of $x$ with the following:
\begin{equation*}
F_*^e k \llbracket x_1,\dots,x_n \rrbracket  \ni \prod_{i=1}^n x_i^{j_i} \longmapsto 
\left\{
\begin{array}{lp{10pt}l}
\prod_{i=1}^n x_i^{\frac{j_i - p^e -1}{p^e}}  \in  k \llbracket x_1,\dots,x_n \rrbracket & &  \textrm{if } p^e | j_i - p^e -1  \quad (\forall i),  \\[8pt]
0 \in  k \llbracket x_1,\dots,x_n \rrbracket  & & \textrm{otherwise}. 
\end{array}
\right.
\end{equation*}
\end{remark}

\begin{proposition}
\label{prop:non_zero_semi_stable_part}
Let $X$ be a projective $S_2$ variety of dimension $n$. Let $\sH$ be an ample line bundle, and let $l>0$ be an integer.  Then, for every integer $s \gg 0$, the following holds: For any integer $p\nmid d >0$ and for general $D \in \left|\sH^{sd} \right|$, we have $\dim_k H^n (X, \sH^{-s} )^{ \sstab,D} \geq l$. 
\end{proposition}

\begin{proof}
Choose an integer $e>0$ such that $d | p^e-1$, and set $r:=\frac{p^e-1}{d}$.

According to \autoref{lem:stabilization_genericity}, we only have to exhibit a single $D$ as above. Additionally, with this single~$D$, we can show that the Serre-dual action has at least $l$-dimensional semi-stable part. Additionally, in degrees~$0$ and $n$, Serre duality works for $S_2$ varieties by the proof of \cite[Proposition~2.4]{Patakfalvi_Zdanowicz_Ordinary_varieties_with_trivial_canonical_bundle_are_not_uniruled}. Hence, we need to show that the following action  on $H^0(X, \omega_X \otimes  \sH^s)$ has at least $l$-dimensional semi-stable part:
\begin{equation}
\label{eq:non_zero_semi_stable_part:action}
\xymatrix@C=30pt{
H^0(X, \omega_X \otimes \sH^s) \ar[r]^(0.35){\cdot rD} & H^0\left(X, \omega_X \otimes \sH^{sp^e} \right)  \cong H^0(X, \sH^s \otimes F^e_* \omega_X ) \ar[rr]^(0.64){H^0\left(X, \Tr_{F^e} \otimes \Id_{\sH^s} \right)} 
& & H^0(X, \omega_X \otimes \sH^s). 
}
\end{equation}
Fix  pairwise-distinct closed points $x_1,\dots,x_l \in X_{\reg}$. For each $1 \leq j \leq l$,  let 
\begin{equation*}
\left\{ t_{i,j} \in m_{X,x_j} \mid i = 1, \dots, n, j=1,\dots, l \right\}
\end{equation*}
 be a regular system of parameters, and define the ideals
$I_j:=\left( \prod_{i=1}^n t_{i,j} \right) \cdot \sO_{X,x} \subseteq \sO_{X,x}$.
 After this, choose an integer $s \gg 0$ satisfying  the following conditions:
 \begin{itemize}
 \item For all $1 \leq j \leq l$, there are sections $g_j \in H^0(X, \omega_X \otimes \sH^s)$ such that 
 \begin{equation}
 \label{eq:non_zero_semi_stable_part:vanishing}
 \left(g_j\right) \otimes k\left(x_{j'}\right) = \left\{
 \begin{matrix}
 1 & \textrm{if } j=j', \\
 0 & \textrm{otherwise}, 
 \end{matrix}
 \right. 
 \end{equation}
 \item There is an $h \in H^0(X,\sH^s)$ such that for all $1 \leq j \leq l$, we have $h_{x_j}  \in I_j \setminus \left( I_j \cdot m_{X,x_j}\right)$. 
 \end{itemize}
 Let $\Gamma \in |\sH^s|$ be the divisor corresponding to $h$, and set $D:= d \Gamma$. The main point is that if we apply the  action of \autoref{eq:non_zero_semi_stable_part:action} to $g_j$, then by the above choice of $D$, this is  the same as applying the trace (or more precisely $H^0\left(X, \Tr_{F^e} \otimes \Id_{\sH^s} \right)$) to $g_j \cdot h^{p^e-1}$. As 
 \begin{equation*}
 \left(g_{j'} \cdot h^{p^e-1}\right)_{x_{j}}  \in
 \left\{
 \begin{matrix}
 I_j^{p^e-1} \setminus I_j^{p^e-1} \cdot m_{X, x_j} & \textrm{if } j'= j, \\[.5ex]
 I_j^{p^e-1} \cdot m_{X, x_j} & \textrm{otherwise}. 
 \end{matrix}
 \right.
 \end{equation*}
using   \autoref{rem:Frobenius_trace_locally}, it follows that trace takes $g_{j} \cdot h^{p^e-1}$ to a section that also satisfies  property  \autoref{eq:non_zero_semi_stable_part:vanishing}. In particular, the same holds for the image of $g_j$ via the  action of \autoref{eq:non_zero_semi_stable_part:action}. Iterating this argument, we obtain that after iterating the action of \autoref{eq:non_zero_semi_stable_part:action} arbitrarily many times, for every $1 \leq j \leq l$,  there will be a section in the image that is non-zero at $x_j$ and  is zero at $x_{j'}$ for every $j \neq j'$. This shows that the image is at least $l$-dimensional after arbitrarily many iterations, concluding our proof. 
\end{proof}

\subsection{General cyclic covers}

\begin{theorem}
\label{thm:non_uniruled_cyclic_cover}
If\, $X$ is a projective, $S_3$  variety of dimension $n$ over $k$ and $\sH$ an ample line bundle on $X$, then for every integer $s \gg 0$,  the following holds: For every integer $ p \nmid d >0$ and for every general $D \in |\sH^{sd}|$, if\, $Y$ is the  corresponding degree $d$ cyclic cover, then $H^n\left(Y, W\sO_{Y,\bQ}\right) \neq 0$.

If additionally $X$ is normal and $Y$ has $W\sO$-rational singularities, then $Y$ is not uniruled for $s \gg 0$. 
\end{theorem}

\begin{proof}
Let $\pi \colon Y \to X$ be the considered cyclic cover. Then, we have $\pi_* \sO_Y \cong \bigoplus_{j=0}^{d-1} \sH^{-js}$, and as $D$ is general, $Y$ is normal. By the proof of \cite[Proposition~2.4]{Patakfalvi_Zdanowicz_Ordinary_varieties_with_trivial_canonical_bundle_are_not_uniruled},
 in degrees $1$ and $n-1$, Serre duality works for $S_3$ varieties. Pairing this up with Serre vanishing, we obtain that for every $s \gg 0$, we have $H^{n-1}\left(X, \sH^{-sj}\right) = 0$ for every integer $j>0$. Hence, for every integer $s \gg 0$, we have $H^{n-1}(Y, \sO_Y) \cong  H^{n-1}(X, \sO_X)$. In particular, according to 
 \autoref{prop:Witt_cohom_non_zero}, it is enough to show that for any integer $l>0$, for every integer $s \gg 0$, we have $H^n(Y, \sO_Y)^{\sstab} \neq 0$, for choices of $d$ and $D$ as in the statement of the theorem. However, it is easy to see that 
\begin{equation*}
H^n(Y, \sO_Y)^{\sstab} =H^n(X, \sO_X)^{\sstab} \oplus \left(  \bigoplus_{j=1}^{d-1} H^n\left(X,\sH^{-sj}\right)^{\sstab, D} \right).
\end{equation*}
So, we are done by \autoref{prop:non_zero_semi_stable_part}.
\end{proof}

In \autoref{thm:non_uniruled_cyclic_cover}, it is expected that it is enough to assume that $X$ has $W \sO$-rational singularities, instead of the current assumption that $Y$ has $W \sO$-rational singularities. The corresponding questions are the following. 

\begin{question}
Suppose that $X$ is a $W \sO$-rational variety over $k$. 
\begin{enumerate}
\item Does a general hyperplane section of $X$ have $W \sO$-rational singularities?
\item Is a general cyclic cover, as in \autoref{thm:non_uniruled_cyclic_cover}, $W \sO$-rational?
\end{enumerate}
\end{question}

\section{Proof of \autoref{thm:main_intro} and \autoref{cor:subadditivity}} 
\label{sec:proof}

\subsection{Lemmas}

\begin{lemma}
\label{lem:uniruled}
Let $f \colon X \to T$ be a surjective projective morphism of  varieties 
such that both $T$ is not uniruled and the geometric generic fiber is integral. If\, $X$ is uniruled, then so is the geometric generic fiber $X_{\overline{\eta}}$ of $f$. 
\end{lemma}

\begin{proof}
First, note that by the definition \cite[Definition IV.1.1]{Kollar_Rational_curves_on_algebraic_varieties}, being uniruled is invariant under finite base-extension of geometrically integral varieties. Also taking into account \cite[Proposition IV.1.3]{Kollar_Rational_curves_on_algebraic_varieties}, one can even drop the word ``finite.''  In particular, we may assume that $k$ is uncountable.

Second, note that by shrinking $T$, we may assume that $f$ is flat. Then \cite[Theorem 1.8.1]{Kollar_Rational_curves_on_algebraic_varieties} tells us that there are countably many subsets $\bigcup_i T_i$ such that $X_t$ is uniruled if and only if $t \not\in T_i$. Hence, it is enough to show that $T_i \neq T$ for all $i$, or equivalently that the very general closed fiber is uniruled. 

As $T$ is not uniruled, according to \cite[Proposition~IV.1.3.6]{Kollar_Rational_curves_on_algebraic_varieties}
we are in the following situation: If $R_i$ is the closure of the image of the cycle map from the universal family over the space of degree $i$ rational curves on $T$,  then the following hold: 
\begin{itemize}
\item $R_i \neq T$ for every $i$. 
\item  There is  no rational curve in $T$ through any $t \in T \setminus \left( \bigcup_i R_i\right)$.
\end{itemize}
Furthermore, according to \cite[Proposition~IV.1.3.5]{Kollar_Rational_curves_on_algebraic_varieties},
there is an open set $U \subseteq X$ such that there is a rational curve through each $x \in U(k)$ in $X$. Hence, for any $k$-point $x \in U \setminus \left(\bigcup_i f^{-1} R_i \right)$,  there is a rational curve $C_x$ through $x$, but $f(C_x)$ cannot give a rational curve through $f(x)$. Therefore, $C_x \subseteq X_{f(x)}$ must hold. This shows that for any $k$-point $t \in   T \setminus \left( \bigcup_i R_i \right)$, there is a rational curve in $X_t$ through every closed point of $U \cap X_t$. Additionally,  for general such $t$, we have $U \cap X_t \neq \emptyset$. Hence, we obtain that very general fibers of $f$ are uniruled.
\end{proof}

\begin{lemma}
\label{lem:pullback_variety}
Let $f \colon X \to T$ be a surjective  morphism of  projective varieties with integral geometric generic fiber, and let $S \to T$ be a finite flat morphism of varieties. Then $X \times_T S$ is a projective variety $($that is, it is integral\,$)$. 
\end{lemma}

\begin{proof}
By the geometric generic fiber assumption, the generic fiber of $X \times_T S \to S$ is integral; see \cite[Exercise~II.3.15]{Hartshorne_Algebraic_geometry}. Hence, $X \times_T S$ is integral at its generic point, and then it is enough to see that $X \times_T S$ satisfies the $S_1$ property. Indeed,  $S_1$ varieties have no embedded points by the definition of the $S_1$ property.

However, the fact that $X \times_T S$ is $S_1$ follows directly from the behavior of the depth along flat maps. In the special case of finite flat morphisms, which is $X \times_T S \to X$ in our case, the depth is simply the same for points lying over each other; see \cite[Proposition~6.3.1]{Grothendieck_Elements_de_geometrie_algebrique_IV_II}. This concludes our proof. 
\end{proof}

\subsection{The proof}
\label{ssec:proof}

Recall that a morphism of varieties is \emph{separable} if the geometric generic fiber is reduced. Also, by the definition in \autoref{sec:local_compl_int}, local complete intersection singularities are Gorenstein. According to point \autoref{itm:complete_intersection_singularities:relative_absolute} of \autoref{lem:complete_intersection_singularities}, if $f\colon X \to T$ is a surjective local complete intersection morphism from a variety to a smooth variety, $X$ also has  local complete intersection singularities. Hence, in this case $X$ is Gorenstein. Recall additionally that  the general definitions of $\omega_{X/T}$ and $\omega_X $ are $\omega_{X/T}:=\sH^{-\dim X/T}\left(f^! \sO_T\right)$ and $\omega_{X}:=\sH^{-\dim X}\left(g^! \sO_{\Spec (k)}\right)$, where $g \colon X \to \Spec (k)$ is the structure morphism.  Therefore, whenever $T$ is Gorenstein, and hence $\omega_T$ is a line bundle,  by \cite[Theorem 5.4]{Neeman_The_Grothendieck_duality_theorem}, we have
\begin{equation}
\label{eq:rel_abs_canonical}
\omega_{X/T} = \omega_X \otimes f^* \omega_T^{-1}.
\end{equation}
 In the above special situation,  however, we also know that both sides of \autoref{eq:rel_abs_canonical} are  line bundles, by the Gorenstein property. 
In particular, in this case, there is a well-defined linear equivalence class $K_{X/T}$ of Cartier divisors corresponding to $\omega_{X/T}$, even if $X$ is not normal, satisfying $K_X = K_{X/T} + f^* K_T$. This is important to make sense of the statement of \autoref{thm:main}. It is also important to note that if $S \to T$ is a flat morphism, then $\omega_{X_S/S} \cong \tau^* \omega_{X/T} $, where $X_S:=X \times_T S$ and $\tau  \colon X_S \to X$ are the induced morphisms; see \cite[Theorem III.8.7(5)]{Hartshorne_Residues_and_duality}.  Therefore, one has the following base-change isomorphism on the level of divisors: 
\begin{equation}
\label{eq:base_change}
K_{X_S/S} \cong \tau^* K_{X/T}.
\end{equation}

\begin{theorem}
\label{thm:main}
Let $f \colon X \to T$ be a surjective, local complete intersection, separable, projective morphism to a smooth projective variety $T$, such that the geometric generic fiber is connected and not uniruled. 
Then $K_{X/T}$ is pseudo-effective. 
\end{theorem}

\begin{proof}
Assume that $K_{X/T}$ is not pseudo-effective. We will derive a contradiction.

Recall that a Cartier divisor is  pseudo-effective if and only if its pullback under  an arbitrary finite morphism is pseudo-effective. Hence, using \autoref{lem:pullback_variety} and \autoref{eq:base_change},  we may replace $f$ by a pullback $f_S$ via any finite flat morphism $S \to T$ between smooth varieties. In particular, according to \autoref{thm:non_uniruled_cyclic_cover}, we may assume that $T$ is not uniruled.  

Let $F^n \colon T^n \to T$ be the $n$-times iterated Frobenius morphism (for all $n\geq 0$), and set $X^n:=X \times_T T^n$. Then, by \autoref{lem:uniruled}, $X^n$ are not uniruled.

If $K_{X/T}$ is not pseudo-effective, then there is a general element $C$ of a moving family of curves on $X$ with $K_{X/T} \cdot C <0 $; see  \cite[Theorem 1.4]{Das_Finiteness_of_Log_Minimal_Models_and_Nef_curves_on_3-folds_in___characteristic_p>5} and \cite[Remark 2.1]{Fulger_Seshadri_constants_for_curve_classes}. According to \cite[Theorem 1.4]{Das_Finiteness_of_Log_Minimal_Models_and_Nef_curves_on_3-folds_in___characteristic_p>5},  we may even assume that $C$ is irreducible.

 Let $C^n$ be the normalization of the reduced preimage  of $C$ in $X^n$. We note that as $\sigma^n \colon X^n \to X$ is inseparable, $C^n \to C$ is a normalization followed by a few Frobenii. In particular, $\sigma_*^n C^n = a_n C$ for some integer $a_n>0$.
 
 As $C$ is general in a moving family, $X^n$ is smooth along the general points of $C^n$, and it has complete intersection singularities according to point \autoref{itm:complete_intersection_singularities:relative_absolute} of
  \autoref{lem:complete_intersection_singularities}. Hence \cite[Theorem IV.5.14 and Remark IV.5.15]{Kollar_Rational_curves_on_algebraic_varieties} apply to $C^n$ and $X^n$ as soon as we know that $K_{X^n} \cdot C^n <0$. However, for $n \gg 0$ this is satisfied because of the following, where $f^n \colon X^n \to T^n$ is the induced morphism:
\begin{multline*}
K_{X^n} \cdot C^n = \left( K_{X^n/T^n} + \left( f^n\right)^* K_{T^n} \right) \cdot C^n 
\expl{\equiv}{\autoref{eq:base_change} and the fact that $\left(F^n \right)^* K_T \sim p^n K_T$}
\left( \left( \sigma^n \right)^* K_{X/T}    + \frac{ \left( f^n\right)^* \left( F^n\right)^* K_{T} }{p^n} \right) \cdot C^n
\\ 
\explshift{50pt}{=}{$F^n \circ f^n = f \circ \sigma^n$} 
\left( \sigma^n \right)^* \left(  K_{X/T}   + \frac{ f^* K_{T}  }{p^n} \right) \cdot  C^n 
%
%
\explshift{-40pt}{=}{projection formula}
%
%
\left( K_{X/T}   + \frac{f^* K_{T} }{p^n} \right)  \cdot \sigma^n_* C^n
%
%
 \explshift{-50pt}{=}{$\sigma^n_* C^n = a_n C$}
 \left( K_{X/T}  + \frac{f^* K_{T}}{p^n} \right)  \cdot a_n C  
\explshift{-90pt}{<0.}{for $n \gg 0$, as $K_{X/T} \cdot C < 0$}
\end{multline*}
So, according to \cite[Theorem 5.14 and Remark 5.15]{Kollar_Rational_curves_on_algebraic_varieties}, for every $n \gg 0$,  there is a rational curve through each point of $C^n$. As $C$ is general, there is a rational curve through a general point of $X^n$. Hence, $X^n$ is uniruled, which gives a contradiction. Hence, our assumption that $K_{X/T}$ is not pseudo-effective was false. 
\end{proof}

\begin{proof}[Proof of \autoref{thm:main_intro}]
This follows immediately from \autoref{thm:main} using point \autoref{itm:complete_intersection_singularities:regular_over_regular} of \autoref{lem:complete_intersection_singularities}.
\end{proof}

\begin{proof}[Proof of \autoref{cor:subadditivity}]
  As $K_{X/T}$ is $f$-big, we may write $K_{X/T} \sim_{\bQ} A + E$, where $A$ is an $f$-ample and $E$ is an effective $\bQ$-Cartier $\bQ$-divisor on $X$. Similarly, we may write $K_T = H + G$, where $H$ is an ample and $G$ is an effective $\bQ$-divisor on $T$. Hence, $K_X$ is big by the following computation, which concludes our proof:
\begin{equation*}\pushQED{\qed}
K_X = f^* K_T + K_{X/T} = f^* G+f^* H  + \varepsilon(A +E) + (1- \varepsilon) K_{X/T} = 
\explshift{-220pt}{\underbracket{\varepsilon A + f^* H}}{ample for $0 < \varepsilon \ll 1$, as $A$ is $f$-ample, and $H$ is ample}
+
\explshift{-20pt}{\underbracket{f^*G + \varepsilon E + (1- \varepsilon) K_{X/T}}}{pseudo-effective by \autoref{thm:main_intro}}\qedhere \popQED
\end{equation*}
\renewcommand{\qed}{}     
\end{proof}


\end{document}